\documentclass[11pt,a4paper]{article}
\usepackage{amsmath}
\usepackage{amssymb}
\usepackage{amsthm}
\usepackage{amsfonts}
\usepackage{enumerate}
\usepackage{xspace} %for Matlab,Intlab,NP - \xspace
\usepackage{graphics} %cx na konci pro PDF
\usepackage{graphicx}
\usepackage{pgfpages}
\usepackage{url}
\usepackage{booktabs}
\usepackage{verbatim}
\usepackage{hyperref}
\usepackage{breakurl}
\usepackage[margin=1.3in]{geometry}
\newcommand{\tluste}[1]{\mbox{\mathversion{bold}$ #1 $}}

\newcommand{\mace}[1]{{{#1}}}
\newcommand{\mna}[1]{{\mathcal{#1}}}

\newcommand{\omace}[1]{\mbox{$\overline{\mace{#1}}$}} 
\newcommand{\umace}[1]{\mbox{$\underline{\mace{#1}}$}} 
\newcommand{\imace}[1]{\mbox{$\tluste{#1}$}}

\def\Mid#1{#1_c}
\def\Rad#1{#1_\Delta}
\def\Mag#1{\mathop{\mathrm{mag}}#1}	%magnitude of interval
\def\Mig#1{\mathop{\mathrm{mig}}#1}	%mignitude of interval
\def\comp#1{{\langle{#1}\rangle}}%comparison matrix

\newcommand{\onum}[1]{\mbox{$\overline{{#1}}$}} 
\newcommand{\unum}[1]{\mbox{$\underline{{#1}}$}}

\newcommand{\ivr}[1]{\mbox{$\tluste{#1}$}} 

\newcommand{\inum}[1]{\mbox{$\tluste{#1}$}} 
\newcommand{\R}[0]{{\mathbb{R}}}
\newcommand{\C}[0]{{\mathbb{C}}}
\newcommand{\IR}[0]{{\mathbb{IR}}}

\newcommand{\mmid}[0]{;\,}		%pouzivejte v definici mnozin!

\def\clqq{``}
\def\crqq{''}
\def\quo#1{\clqq{}#1\crqq{}}  % snadny zapis ang. uvozovek

\DeclareMathOperator{\diag}{diag}	%diagonal matrix

\def\nref#1{$(\ref{#1})$}

\newtheorem{theorem}{Theorem}

\newtheorem{conjecture}{Conjecture}
\theoremstyle{definition}

\newtheorem{definition}{Definition} 

\begin{document}

\title{AE regularity of interval matrices}

%\author{Milan Hlad\'{i}k \and Evgenija D. Popova}
\author{
  Milan Hlad\'{i}k\footnote{
Charles University, Faculty  of  Mathematics  and  Physics,
Department of Applied Mathematics, 
Malostransk\'e n\'am.~25, 11800, Prague, Czech Republic, 
e-mail: \texttt{milan.hladik@matfyz.cz}
}
}

\date{\today}
\maketitle

\begin{abstract}
Consider a linear system of equations with interval coefficients, and each interval coefficient is associated with either a universal or an existential quantifier. The AE solution set and AE solvability of the system is defined by $\forall\exists$-quantification. 
%That is, a vector $x$ is an AE solution if for every realization of $\forall$-coefficients there is a realization of $\exists$-coefficients such that $x$ solves the corresponding system.
Herein, we deal with the problem what properties must the coefficient matrix have in order that there is guaranteed an existence of an AE solution. 
Based on this motivation, we introduce a concept of AE regularity, which implies that the AE solution set is nonempty and the system is AE solvable for every right-hand side. We discuss characterization of AE regularity, and we also focus on various classes of matrices that are implicitly AE regular. Some of these classes are polynomially decidable, and therefore give an efficient way for checking AE regularity. We also state open problems related to computational complexity and characterization.
\end{abstract}

\textbf{Keywords:}\textit{ Interval computation, quantified systems, linear equations, interval systems.}

%%%%%%%%%%%%%%%%%%%%%%%%%%%%%%%%%%%%%%%%%%%%%%%%%%%%%%%%%%%%%%% 
% INTRODUCTION
%%%%%%%%%%%%%%%%%%%%%%%%%%%%%%%%%%%%%%%%%%%%%%%%%%%%%%%%%%%%%%% 
\section{Introduction}

Solving systems of interval linear equations is a basic problem of interval computation \cite{MooKea2009,Neu1990}. In the last decade, there was a particular interest in the so called AE solutions defined by $\forall\exists$ quantification of interval parameters. AE solutions were studied not only for interval systems  \cite{Gol2005,Hla2015a,LiXia2017,Pop2012,PopHla2013,Sha2002}, but also in the context of linear programming \cite{Hla2016c,Li2015,LiLiu2015,LuoLi2014a}. The purpose of this paper is to investigate what properties the interval matrix should have such that the interval system is AE solvable. We will denote this property AE regularity.

An interval matrix is defined as
$$
\imace{A}:=\{A\in\R^{m\times n}\mmid \umace{A}\leq A\leq \omace{A}\},
$$
where $ \umace{A}$ and $\omace{A}$, $\umace{A}\leq\omace{A}$, are given matrices, and the inequality between matrices is understood entrywise. The corresponding midpoint and the radius matrices are defined respectively as
$$
\Mid{A}:=\frac{1}{2}(\umace{A}+\omace{A}),\quad
\Rad{A}:=\frac{1}{2}(\omace{A}-\umace{A}).
$$
The set of all $m\times n$ interval matrices is denoted by $\IR^{m\times n}$, and intervals and interval vectors are considered as special cases of interval matrices.

Consider an interval system of linear equations $\imace{A}x=\ivr{b}$, where $\imace{A}\in\IR^{m\times n}$ and $\ivr{b}\in\R^m$. Its solutions set is traditionally defined as the union of all solutions of realizations of interval coefficients, that is
$$
\{x\in\R^n\mmid \exists A\in\imace{A},\,\exists b\in\ivr{b}:Ax=b\}.
$$
We say that $\imace{A}$ is \emph{regular} if every $A\in\imace{A}$ is nonsingular. Regularity of $\imace{A}$ implies that each system realization has a unique solution and the solution set is a bounded polyhedron. Checking whether an interval matrix is regular, however, is a co-NP-hard problem \cite{Fie2006,PolRoh1993}. A survey of forty necessary and sufficient condition for regularity were summarized by Rohn \cite{Roh2009}.

Next, we say that some property $\mna{P}$ holds strongly (weakly) for an interval matrix $\imace{A}$ if it holds for every (some) matrix $A\in\imace{A}$.

Let us now consider a more general concept by using $\forall\exists$ quantification of interval parameters. Each interval of $\imace{A}$ and $\ivr{b}$ is associated either with the universal, or with the existential quantifier. Thus, we can disjointly split the interval matrix as $\imace{A}=\imace{A}^{\forall}+\imace{A}^{\exists}$, where $\imace{A}^{\forall}$ is the interval matrix comprising universally quantified coefficients, and  $\imace{A}^{\exists}$ concerns existentially quantified coefficients.
Similarly, we decompose the right-hand side vector $\ivr{b}=\ivr{b}^{\forall}+\ivr{b}^{\exists}$. Now, $x\in\R^n$ is called an \emph{AE solution} if 
\begin{align*}
\forall A^{\forall}\in\imace{A}^{\forall},
\forall b^{\forall}\in\ivr{b}^{\forall},
\exists A^{\exists}\in\imace{A}^{\exists},
\exists b^{\exists}\in\ivr{b}^{\exists}:\,  
(A^{\forall}+A^{\exists})x=b^{\forall}+b^{\exists}.
\end{align*}
The interval system $\imace{A}x=\ivr{b}$ is called \emph{AE solvable} if for each realization of $\forall$-parameters there are realizations of $\exists$-parameters such that the resulting system has a solution. Formally, it is AE solvable if 
\begin{align*}
\forall A^{\forall}\in\imace{A}^{\forall},
\forall b^{\forall}\in\ivr{b}^{\forall},
\exists A^{\exists}\in\imace{A}^{\exists},
\exists b^{\exists}\in\ivr{b}^{\exists}:\,  
(A^{\forall}+A^{\exists})x=b^{\forall}+b^{\exists} \mbox{ is solvable.}
\end{align*}
Obviously, if the interval system has an AE solution, then it is AE solvable, but the converse implication does not hold in general \cite{Hla2015a}.

Related to AE solvability, we introduce the following natural concept of regularity.

\begin{definition}
An interval matrix $\imace{A}=\imace{A}^{\forall}+\imace{A}^{\exists}$ is called \emph{AE regular} if
$\forall A^{\forall}\in \imace{A}^{\forall} \exists A^{\exists}\in \imace{A}^{\exists}$ such that $A=A^{\forall}+A^{\exists}$ is nonsingular.
\end{definition}

AE regularity generalizes regularity of $\imace{A}$ (which is the case when there ar no $\exists$- parameters), so it is also co-NP-hard to check this property. It is also the question how to characterize AE regularity. In the following section, we will approach to this problem and show classes of matrices that are implicitly AE regular.

%%%%%%%%%%%%%%%%%%%%%%%%%%%%%%%%%%%%%%%%%%%%%%%%%%%%%%%%%%%%%%% 
\section{Strong singularity}

An interval matrix in the form $\imace{A}=0+\imace{A}^{\exists}$, that is, with no $\forall$-quantified interval parameters, is AE regular if and only if there is at least one nonsingular matrix in $\imace{A}$. The negation of this property is an interval matrix containing only singular matrices. Even though this property is not the typical case, it may happen. If we want to characterize AE regularity for the general case, we have to inspect also this particular situation.

Recall that an interval matrix $\imace{A}$ is \emph{strongly singular} if every $A\in\imace{A}$ is singular. 
In the following, we use the term \emph{a vertex} matrix of $\imace{A}$, which is any matrix $A\in\imace{A}$ such that $a_{ij}\in\{\unum{a}_{ij},\onum{a}_{ij}\}$ for all $i,j$.

\begin{theorem}\label{thmSingVert}
$\imace{A}$ is strongly singular if and only if each vertex matrix is singular.
\end{theorem}

\begin{proof}
$\imace{A}$ is strongly singular if and only if and only if $\det(A)=0$ for every $A\in\imace{A}$. Due to linearity of the determinant with respect to the $i,j$th entry, the largest and the lowest value of determinants are attained for vertex matrices.
\end{proof}

\begin{theorem}
If $\imace{A}$ is strongly singular, then $[-\Rad{A},\Rad{A}]$ is strongly singular and in particular $\Rad{A}$ is singular.
\end{theorem}

\begin{proof}
Let $A^1\in\imace{A}$ be a matrix that results from $\Mid{A}$ by replacing the first row by $\omace{A}_{1*}$. Since $\det(A^1)=\det(\Mid{A})=0$ and by row linearity of determinants, the matrix $A^2=A^1-\Mid{A}$ is singular. This matrix has radii in the first row and midpoints in the others. Similarly, we can show singularity of a matrix $A^3$ having radii in the first row, right endpoints in the second and midpoints in the others. By row linearity of determinants we again have that $A^3-A^2$ is singular. This matrix has radii in the first two rows and midpoints in the others. Proceeding further, we arrive at singularity of $\Rad{A}$. 

Replacing $\omace{A}$ by any other matrix $A\in\imace{A}$ in the above considerations, we obtain singularity of any matrix in $[-\Rad{A},\Rad{A}]$.
\end{proof}

Matrices of type $A_{yz}=\Mid{A}-\diag(y)\Rad{A}\diag(z)$, where $y,z\in\{\pm1\}^n$ and $\diag(y)$ denotes the diagonal matrix with $y$ on the diagonal, are often used in verifying various properties of interval matrices such as positive definiteness, and in some sense also regularity \cite{Fie2006}. Notice however that strong singularity cannot be checked by inspecting only these matrices as the following counterexample shows. The matrix
$$
\imace{A}=\begin{pmatrix}[-1,1]&[-1,1]\\{}[-1,1]&[-1,1]\end{pmatrix}
$$
is not strongly singular, but all matrices of type $A_{yz}$ are singular.

We leave two important open questions here:
\begin{itemize}
\item
Is there a simpler (computationally cheaper) characterization of strong singularity?
\item
What is the computational complexity of checking strong singularity. Is in a polynomial or an NP-hard problem?
\end{itemize}

As an open problem we also state the following conjecture. The \quo{if} part is obvious, but the converse is the open and hard one.

\begin{conjecture}
$\imace{A}$ is strongly singular if and only if it has a submatrix of size $k\times\ell$ that is real and has the rank of $k+\ell-n-1$.
\end{conjecture}

%%%%%%%%%%%%%%%%%%%%%%%%%%%%%%%%%%%%%%%%%%%%%%%%%%%%%%%%%%%%%%% 
\section{AE regularity}

Now, we consider AE regularity in the general form. We firs show this property is useful for AE solvability of interval systems.

\begin{theorem}
If $\imace{A}$ is AE regular, then $\imace{A}x=\ivr{b}$ is AE solvable for each $\ivr{b}$. The converse is not true in general.
\end{theorem}

\begin{proof}
For each $\forall$-realization, we find $\exists$-realization such that $A$ is nonsingular, whence solvability of $Ax=b$ follows.

The counter-example for the converse direction is
\begin{equation*}
\imace{A}^{\forall}=0,\quad
\imace{A}^{\exists}=\begin{pmatrix}0&[-1,1]\\ 0&[-1,1]\end{pmatrix}.
%,\quad b=0.
%\qedhere
\end{equation*}
Then $\imace{A}$ is not AE regular, but $\imace{A}x=\ivr{b}$ is AE solvable. 
\end{proof}

Obviously, AE regularity implies nonemptiness of the AE solution set. On the other hand, AE regularity does not imply boundedness of the AE solution set; counter-example: 
$
\imace{A}^{\forall}=0,\ 
\imace{A}^{\exists}=([-1,1]),\ \ivr{b}=0.
$

In order to characterize AE regularity, we utilize the following master interval linear system with linear dependencies given by multiple appearance of ${A}^{\forall}$
\begin{align}\label{eqThmCharAeReg}
(A^{\forall}+A^{\exists}_{v})x_{v}=0,\ \ 
[-e,e]^Tx_{v}=1,\ \ 
v\in V,
\end{align}
where $A^{\forall}\in\imace{A}^{\forall}$ and $A^{\exists}_{v}$, $v\in V$, are all vertex matrices of $\imace{A}^{\exists}$.

\begin{theorem}
$\imace{A}$ is not AE regular if and only if \nref{eqThmCharAeReg} is solvable for some $A^{\forall}\in\imace{A}^{\forall}$.
\end{theorem}

\begin{proof}
$\imace{A}$ is not AE regular if and only if there is $A^{\forall}\in\imace{A}^{\forall}$ such that $A^{\forall}+\imace{A}^{\exists}$ strongly singular. By Theorem~\ref{thmSingVert}, this is equivalent to the condition that all matrices $A^{\forall}+A^{\exists}_{v}$, $v\in V$, are singular. The system \nref{eqThmCharAeReg} then formulates singularity of these matrices.
\end{proof}

By the above theorem, we reduced AE regularity to solvability of a parametric system. Parametric systems are hard to solve and characterize, even for particular case; see \cite{Hla2008g,May2012,Pop2012}. This suggests that also AE solvability is a very hard problem in general.

%%%%%%%%%%%%%%%%%%%%%%%%%%%%%%%%%%%%%%%%%%%%%%%%%%%%%%%%%%%%%%% 
\section{Special classes}

This section presents several classes of interval matrices that are inherently AE regular.

%%%
\subsection{M-matrix}

A real matrix $A\in\R^{n\times n}$ is an M-matrix if the off-diagonal entries are non-positive and there is $x>0$ such that $Ax>0$.
Interval M-matrices in the strong sense (i.e., with $\forall$-quantification) were investigated, e.g., in
\cite{AleHer1983,Neu1990}. We first discuss $\exists$-quantified version, and then extend it to the general $\forall\exists$ case.

We say that $\imace{A}$ is weakly an M-matrix if there is $A\in\imace{A}$ being an M-matrix. Weak M-matrices are characterized as follows.

\begin{theorem}\label{thmWeakMmat}
Define $\tilde{A}\in\imace{A}$ as follows
\begin{align}\label{eqThmWeakMmat}
\tilde{a}_{ij}=
\begin{cases}
\onum{a}_{ij} & \mbox{if } i=j,\\
\arg\min\{|{a}_{ij}|\mmid {a}_{ij}\in\inum{a}_{ij}\} & \mbox{if } i\not=j.
\end{cases}
\end{align}
Then $\imace{A}$ is weakly an M-matrix if and only if $\tilde{A}$ is an M-matrix.
\end{theorem}

\begin{proof}
\quo{If.}
This is obvious as $\tilde{A}\in\imace{A}$.

\quo{Only if.}
Let $A\in\imace{A}$ be an M-matrix. Then there is a vector $v>0$ such that $Av>0$. Further, from $a_{ij}\leq0$ for $i\not=j$ we have that $A\leq\tilde{A}$ and $\tilde{a}_{ij}\leq0$ for $i\not=j$. Eventually, from $\tilde{A}v\geq Av>0$ it follows that $\tilde{A}$ is an M-matrix, too.
\end{proof}

We say that $\imace{A}^{\forall\exists}$ is an AE M-matrix if $\forall A^{\forall}\in \imace{A}^{\forall}\, \exists A^{\exists}\in \imace{A}^{\exists}$ such that $A^{\forall}+ A^{\exists}$ is an M-matrix.
Obviously, an AE M-matrix is AE regular.

\begin{theorem}
Denote by $\tilde{A}=\umace{A}^{\forall}+\tilde{A}^{\exists}$ the matrix from \nref{eqThmWeakMmat} corresponding to the interval matrix $\umace{A}^{\forall}+\imace{A}^{\exists}$.
Then $\imace{A}^{\forall\exists}$ is an AE M-matrix if and only if $\tilde{A}$ is an M-matrix and $(\omace{A}^{\forall}+\tilde{A}^{\exists})_{ij}\leq0$ for all $i\not=j$.
\end{theorem}

\begin{proof}
\quo{If.}

If $\tilde{A}$ is an M-matrix, then there is a vector $v>0$ such that $\tilde{A}v>0$. Hence for every $A^{\forall}\in \imace{A}^{\forall}$ we can take $\tilde{A}^{\exists}$, and for the matrix $A=A^{\forall}+\tilde{A}^{\exists}$ we have $Av\geq\tilde{A}v>0$. Since $A_{ij}\leq(\omace{A}^{\forall}+\tilde{A}^{\exists})_{ij}\leq0$, this part is proved.

\quo{Only if.}
By the assumption, $\umace{A}^{\forall}+\imace{A}^{\exists}$ is weakly an M-matrix, and hence by Theorem~\ref{thmWeakMmat}, $\tilde{A}$ is an M-matrix.
The condition  $(\omace{A}^{\forall}+\tilde{A}^{\exists})_{ij}\leq0$ for all $i\not=j$ holds also from the assumption since from the disjunction of interval parameters we have either $(\omace{A}^{\forall}+\tilde{A}^{\exists})_{ij}=\omace{A}^{\forall}_{ij}$ or $(\omace{A}^{\forall}+\tilde{A}^{\exists})_{ij}=\tilde{A}^{\exists}_{ij}$ and for both cases it is true.
\end{proof}

%%%
\subsection{H-matrix}

A matrix $A\in\R^{n\times n}$ is called an H-matrix, if the so called comparison matrix $\langle A\rangle$ is an M-matrix, where $\langle A\rangle_{ii}=|a_{ii}|$ and $\langle A\rangle_{ij}=-|a_{ij}|$ for $i\not=j$.
Interval H-matrices (corresponding to $\forall$-quantification) were investigated, e.g., in \cite{AleHer1983,Neu1990}. We again first discuss $\exists$-quantified version, which will bed then extended to the general $\forall\exists$ case.

We say that $\imace{A}$ is weakly an H-matrix if there is $A\in\imace{A}$ being an H-matrix. This class of matrices as characterized in the following theorem.
Recall that for an interval $\inum{a}\in\IR$ its magnitude and mignitude are respectively defined as
\begin{align*}
\Mag(\inum{a})
 &=\max\{|a|\mmid a\in\inum{a}\}
 =|\Mid{a}|+\Rad{a},\\
\Mig(\inum{a})
 &=\min\{|a|\mmid a\in\inum{a}\}
 =\begin{cases}0&\mbox{if }0\in\inum{a},\\
    \min(|\unum{a}|,|\onum{a}|)&\mbox{otherwise}.
  \end{cases}
\end{align*}

\begin{theorem}\label{thmWeakHmat}
Define $\tilde{A}\in\imace{A}$ as follows
\begin{align}\label{eqThmWeakHmat}
\tilde{a}_{ij}=
\begin{cases}
 \Mag(\inum{a}_{ij}) & \mbox{if } i=j,\\
-\Mig(\inum{a}_{ij}) & \mbox{if } i\not=j.
\end{cases}
\end{align}
Then $\imace{A}$ is weakly an H-matrix if and only if $\tilde{A}$ is an M-matrix.
\end{theorem}

\begin{proof}
\quo{If.}
This is obvious as the matrix $A\in\imace{A}$ corresponding to $\tilde{A}$ (i.e., the entries of $A$ are attained as magnitutes and mignitudes in \nref{eqThmWeakHmat}) is an H-matrix.

\quo{Only if.}
Let $A\in\imace{A}$ be an H-matrix. Then $\comp{A}$ is an M-matrix, and there is a vector $v>0$ such that $\comp{A}v>0$. Therefore $\tilde{A}v\geq\comp{A}v>0$.
\end{proof}

We say that $\imace{A}^{\forall\exists}$ is an AE H-matrix if $\forall A^{\forall}\in \imace{A}^{\forall}\, \exists A^{\exists}\in \imace{A}^{\exists}$ such that $A^{\forall}+ A^{\exists}$ is an H-matrix.
Obviously, an AE H-matrix is AE regular.

\begin{theorem}
Define the matrix $\tilde{A}$ as follows
\begin{align*}
\tilde{a}^{\exists}_{ij}=
\begin{cases}
 \Mig(\inum{a}^{\forall}_{ij})+\Mag(\inum{a}^{\exists}_{ij})&\mbox{if } i=j,\\
-\Mag(\inum{a}^{\forall}_{ij})-\Mig(\inum{a}^{\exists}_{ij})&\mbox{if } i\not=j.
\end{cases}
\end{align*}
Then $\imace{A}^{\forall\exists}$ is an AE H-matrix if and only if $\tilde{A}$ is an M-matrix.
\end{theorem}

\begin{proof}
\quo{If.}
Let $\tilde{A}^{\forall}\in \imace{A}^{\forall}$ and $\tilde{A}^{\exists}\in \imace{A}^{\exists}$ be the corresponding matrices, for which $\tilde{A}$ is attained. Then $\tilde{A}=\tilde{A}^{\forall}+\tilde{A}^{\exists}$.
Let $A^{\forall}\in \imace{A}^{\forall}$ be arbitrary and define $A:=A^{\forall}+\tilde{A}^{\exists}$. Since $\tilde{A}$ is an M-matrix, there is a vector $v>0$ such that $\tilde{A}v>0$. From $\comp{A}v\geq\tilde{A}v>$ we conclude that $A$ is an H-matrix.

\quo{Only if.}
By the assumption, there exists $A^{\exists}\in \imace{A}^{\exists}$ such that $\tilde{A}^{\forall}+A^{\exists}$ is an H-matrix. That is, there is a vector $v>0$ such that $\comp{\tilde{A}^{\forall}+A^{\exists}}v>0$. 
From $\tilde{A}v\geq\comp{\tilde{A}^{\forall}+A^{\exists}}v>0$ we have that $\tilde{A}$ is an M-matrix.
\end{proof}

It is known that if $\Mid{A}$ is an M-matrix, then $\imace{A}$ is regular if and only if $\imace{A}$ is strongly H-matrix. For generalized quantification, this statement is no longer valid. Consider, for example, the interval matrix
\begin{align*}
\imace{A}^{\forall\exists}
=\begin{pmatrix}[0.8,1]^{\exists}&-[0,1]^{\forall}\\-1&1\end{pmatrix}.
\end{align*}
Then $\Mid{A}$ is an M-matrix and $\imace{A}^{\forall\exists}$ is AE regular, but it is not an AE H-matrix.

%%%
\subsection{Inverse nonnegative matrices}

A matrix $A\in\R^{n\times n}$ is inverse nonnegative if $A^{-1}\geq0$.
Strong inverse nonnegativity of an interval matrix $\imace{A}$, was studied in \cite{Kut1971,Roh1987,Roh2012a}, among others. For this class a simple characterization exists since $\imace{A}$ is inverse nonnegative if and only if $\umace{A}$ and $\omace{A}$ are inverse nonnegative.

In our $\forall\exists$ quantification, we say that $\imace{A}^{\forall\exists}$ is AE inverse nonnegative if $\forall A^{\forall}\in \imace{A}^{\forall}\, \exists A^{\exists}\in \imace{A}^{\exists}$ such that $A:=A^{\forall}+ A^{\exists}$ is inverse nonnegative.
Obviously, an AE inverse nonnegative matrix is AE regular.

In contrast to $\forall$-quantified case, for the general AE inverse nonnegativity it seems there is no simple characterization. 
As a sufficient condition obtained by reversing the order of quantifiers, we obtain that $\imace{A}^{\forall\exists}$ is AE inverse nonnegative if $\exists A^{\exists}\in \imace{A}^{\exists}\,\forall A^{\forall}\in \imace{A}^{\forall}$ the matrix $A^{\forall}+ A^{\exists}$ is inverse nonnegative. This can be characterized in the following manner: $\exists A^{\exists}\in \imace{A}^{\exists}$ such that both matrices $\umace{A}^{\forall}+ A^{\exists}$ and $\omace{A}^{\forall}+ A^{\exists}$ are inverse nonnegative. How to choose $\exists A^{\exists}\in \imace{A}^{\exists}$ is, however, an open question.

%%%
\subsection{Structured quantifiers position}

Real matrices are often somehow structured. Interval matrices can have a specific structure of interval parameters in addition. In this section we focus on intervals matrices with particular structures.

The following theorem characterizes strong singularity for the case when intervals are situated in one row or one column only. By a suitable permutation of rows and columns, the form of \nref{maceThmRobSingRow} can easily be achieved, where $\ivr{b}\in\IR^k$ and $c$ can possibly be empty.

\begin{theorem}\label{thmAeRegRow}
The square interval matrix
\begin{align}\label{maceThmRobSingRow}
\imace{A}=\begin{pmatrix}B&\ivr{b}\\C&c\end{pmatrix}
\end{align}
with $\Rad{b}>0$ is strongly singular if and only if $(B^T\;C^T)$ or $(C\;c)$ has not full row rank.
\end{theorem}

\begin{proof}
\quo{If}. 
Obvious.

\quo{Only if}.
Suppose that $(B^T\;C^T)$ has full row rank. Then $c$ has length at lest 1 since otherwise we could choose $b\in\ivr{b}$ such that $A$ would be nonsingular. Since $\imace{A}$ is strongly singular and $(B^T\;C^T)$ has full row rank, the vector $(b^T,c^T)$ is linearly dependent on the rows of $(B^T\;C^T)$ for each $b\in\ivr{b}$. In particular, $c^T$ is linearly dependent on the rows of $C^T$. 

Now, suppose to the contrary that $(C\;c)$ has full row rank. Then also $C$ has full row rank since otherwise $c^T$ wouldn't bw linearly dependent on the rows of $C^T$. Thus we can extend $C^T$ to a nonsingular square submatrix of $(B^T\;C^T)$. Without loss of generality assume that it is the right part of $(B^T\;C^T)$. Consider the Laplace expansion of $A$ of the last column. Then the coefficient by $a_{1n}$ (i.e., by $b_1$) is nonzero, and therefore by varying $b_1\in\ivr{b}_1$, the determinant of $A$ cannot be constantly zero. A contradiction with strong singularity of $\imace{A}$.
\end{proof}

The following theorem characterizes AE regularity for the case when intervals are situated in one row or one column only. By a suitable permutation of rows and columns, we can always achieve the form
\begin{align}\label{maceThmAeRegRow}
\imace{A}^{\forall\exists}=
\begin{pmatrix}
\imace{B}^{\forall}&\ivr{b}^{\exists}\\\imace{C}^{\forall}&\imace{c}^{\forall}
\end{pmatrix}
\end{align}
with $\Rad{b^{\exists}}>0$.

\begin{theorem}
The square interval matrix \nref{maceThmAeRegRow} is AE regular if and only if $(\imace{B}^T\;\imace{C}^T)$ and $(\imace{C}\;\imace{c})$ have strongly full row rank.
\end{theorem}

\begin{proof}
By negation, $\imace{A}^{\forall\exists}$ is not AE regular if and only if there are $B^{\forall}\in\imace{B}^{\forall}$, $C^{\forall}\in\imace{C}^{\forall}$ and $c^{\forall}\in\imace{c}^{\forall}$ such that 
\begin{align*}
\begin{pmatrix}
{B}^{\forall}&\ivr{b}^{\exists}\\{C}^{\forall}&{c}^{\forall}
\end{pmatrix}
\end{align*}
is strongly singular. By Theorem~\ref{thmAeRegRow}, $({B}^T\;{C}^T)$ or $({B}\;{c})$ has not full row rank. Therefore, $\imace{A}^{\forall\exists}$ is not AE regular if and only if $(\imace{B}^T\;\imace{C}^T)$ or $(\imace{B}\;\imace{c})$ has not full row rank.
\end{proof}

As a related structured matrix, we have the following:

\begin{theorem}
Let $\imace{B}\in\IR^{n\times k}$ and $\imace{C}\in\IR^{n\times (n-k)}$ with $\Rad{C}>0$.
Then $(\imace{B}^{\forall}\,\imace{C}^{\exists})$ is AE regular if and only if $\imace{B}^{\forall}$ has strongly full column rank.
\end{theorem}

\begin{proof}
\quo{If}. By negation, suppose there is $B^{\forall}\in\imace{B}^{\forall}$ such that $(B^{\forall}\,\imace{C}^{\exists})$ is strongly singular. Since $\Rad{C}>0$, the matrix $B^{\forall}$ must have linearly dependent columns (otherwise there is $C^{\forall}\in\imace{C}^{\forall}$ such that $(B^{\forall}\,C^{\exists})$ is nonsingular).

\quo{Only if}. Obvious.
\end{proof}

As an open problem, we leave a generalization of the above two results.

\begin{conjecture}
The square interval matrix 
\begin{align*}
%\imace{A}^{\forall\exists}=
\begin{pmatrix}
\imace{B}^{\forall}&\ivr{D}^{\exists}\\\imace{C}^{\forall}&\imace{E}^{\forall}\end{pmatrix},
\end{align*}
where $\Rad{D^{\exists}}>0$, is AE regular if and only if $(\imace{B}^T\;\imace{C}^T)$ and $(\imace{C}\;\imace{E})$ have strongly full row rank.
\end{conjecture}

%%%%%%%%%%%%%%%%%%%%%%%%%%%%%%%%%%%%%%%%%%%%%%%%%%%%%%%%%%%%%%% 
\section{Conclusion}

We introduced a generalized concept of regularity of interval matrices based on $\forall\exists$ quantification. Characterization of the general case turned out to be a very difficult problem, and we stated several open question. On the other hand, we identified a couple of polynomially recognizable sub-classes such as M-matrices, H-matrices or matrices with structured quantifier position.

%\subsubsection*{Acknowledgments.} 
%The author was supported by the Czech Science Foundation Grant P402/13-10660S.

%%%%%%%%%%%%%%%%%%%%%%%%%%%%%%%%%%%%%%%%%%%%%%%%%%%%%%%%%%%%%%% 
% REFERENCES
%%%%%%%%%%%%%%%%%%%%%%%%%%%%%%%%%%%%%%%%%%%%%%%%%%%%%%%%%%%%%%% 

\end{document}